\documentclass[11pt]{amsart}
\usepackage{amssymb,latexsym}

\def\triangle{\Delta}

\def\be1{{\begin{equation}}}
\def\ee1{{\end{equation}}}

\def\part{\partial}

\def\ba{\begin{array}}
\def\ea{\end{array}}

\newtheorem{them}{Theorem}[section]

\newtheorem{corollary}{Corollary}[section]

\newtheorem{definition}[them]{Definition}

\numberwithin{equation}{section}

\newtheorem{lemma}{Lemma}[section]
\newtheorem{proposition}[lemma]{Proposition}
\newtheorem{theorem}[lemma]{Theorem}
\newtheorem{remark}[lemma]{Remark}

\title[Lagrangian F-stability of closed Lagrangian self-shrinkers]
{Lagrangian F-stability of closed Lagrangian self-shrinkers}

\author{Jiayu Li and Yongbing Zhang}

\address{School of Mathematical Sciences\\
         USTC\\
         Hefei, 230026, Anhui Province, China.}

\email{jiayuli@ustc.edu.cn, ybzhang@amss.ac.cn}

\begin{document}
\maketitle

\begin{abstract}
In this paper, we study the Lagrangian F-stability of closed Lagrangian self-shrinkers immersed in complex Euclidean space.
We show that any closed Lagrangian self-shrinker with first Betti number greater than one is Lagrangian F-unstable.
In particular, any two-dimensional embedded closed Lagrangian self-shrinker is Lagrangian F-unstable.
For a closed Lagrangian self-shrinker with first Betti number equal to one,
we show that Lagrangian F-stability is equivalent to Hamiltonian F-stability.
We also characterize Hamiltonian F-stability of a closed Lagrangian self-shrinker by its spectral property of the twisted Laplacian.
\end{abstract}

\section{Introduction}

By a self-shrinker, we mean an immersed surface in Euclidean space
whose mean curvature vector is related to the normal part of the position vector by
\begin{equation}\label{selfshrinker1}
H=-c(x-x_0)^\perp, \quad c>0.
\end{equation}
A self-shrinker gives rise to a homothetically shrinking solution to the mean curvature flow.
The geometric object becomes important since Huisken's monotonicity formula \cite{Huisken} tells that
any time-slice of a tangent flow (cf. Section \ref{section2}) at a Type-I singularity is a self-shrinker.
The tangent flow at a general singularity is a homothetically shrinking weak solution of the mean curvature flow
(Brakke flow \cite{Brakke}), see \cite{Ilmanen,White}.

Abresch and Langer classified all immersed closed self-shrinkers in the plane.
However even for the case of two dimensional self-shrinkers in $\mathbb{R}^3$,
various examples are expected and a complete classification seems impossible, see for instance \cite{AICh}.
Besides using point-wise conditions, for example mean convexity \cite{Huisken, Huisken2},
recently Colding and Minicozzi \cite{CM} employing a new kind stability to classify self-shrinkers.
That is the entropy-stability for self-shrinkers.

For an $n$-dimensional immersed surface $\Sigma\hookrightarrow \mathbb{R}^N$,
Colding and Minicozzi \cite{CM} introduced the entropy of $\Sigma$ by
\begin{equation}\label{entropy}
\lambda(\Sigma)=\sup_{x_0\in \mathbb{R}^N,t_0>0}\int_\Sigma(4\pi t_0)^{-\frac{n}{2}}e^{-\frac{|x-x_0|^2}{4t_0}}d\mu.
\end{equation}
The entropy $\lambda$ has very nice properties. For example, it is invariant under dilations and rigid motions.

A self-shrinker is called entropy-stable if it is a local minimum of the entropy functional.
The entropy-stability is closely related to the singular behavior of the mean curvature flow, based on the fact that
$\lambda$ is non-increasing along the mean curvature flow in Euclidean space.
In fact entropy-stable self-shrinkers are considered as generic singularities of mean curvature flow in \cite{CM},
i.e. those can not be perturbed away.

In order to classify entropy-stable self-shrinkers, Colding and Minicozzi \cite{CM}
introduced the notion of F-stability for self-shrinkers.
The F-functional with respect to $x_0\in \mathbb{R}^N, t_0>0$ of an immersed surface $\Sigma$ is defined by
\begin{equation}\label{Ffunctional}
F_{x_0,t_0}(\Sigma)=\int_\Sigma(4\pi t_0)^{-\frac{n}{2}}e^{-\frac{|x-x_0|^2}{4t_0}}d\mu.
\end{equation}
Note that the entropy $\lambda(\Sigma)$ is the supremum of the F-functional taken over all $x_0\in \mathbb{R}^N, t_0>0$.
A critical point of $F_{x_0,t_0}$ is self-shrinkers that becomes extinct at $(x_0,t_0)$ (cf. Section 2), and
such a self-shrinker $\Sigma$ is called F-stable if for any $1$-parameter family of deformations $\Sigma_s$ of $\Sigma=\Sigma_0$,
there exist deformations $x_s$ of $x_0$ and $t_s$ of $t_0$ such that
$(F_{x_s,t_s}(\Sigma_s))''\geq 0$ at $s=0$.
Roughly speaking, a self-shrinker is critical point of the entropy functional,
and an F-stable self-shrinker is a self-shrinker at which the second variation of entropy $\lambda$ is non-negative.

The F-stability are closely related to entropy-stability,
and the classification of entropy-stable self-shrinkers relies on the classification of F-stable self-shrinkers.
Colding and Minicozzi \cite{CM} showed that shrinking spheres, cylinders and planes are the only codimension one
F-stable (equivalently, entropy-stable) self-shrinkers.

In this paper we carry over some of Colding and Minicozzi's ideas to the Lagrangian mean curvature flow case.
In particular we shall explore the Lagrangian F-stability of closed Lagrangian self-shrinkers.
We assume our Lagrangian self-shrinkers are closed, orientable and have dimensions greater or equal to two.

An immersed surface $\Sigma^n$ in $\mathbb{C}^n$ is called Lagrangian if the standard K\"ahler form $\overline{\omega}$ of $\mathbb{C}^n$
restricted to $\Sigma$ is zero, or equivalently the standard complex structure of $\mathbb{C}^n$ maps any tangent vector of $\Sigma$ to a normal vector.
When the ambient space is released to K\"ahler-Einstein manifolds, Smoczyk \cite{Smoczyk1}
showed that along the mean curvature flow the Lagrangian condition is preserved, i.e. a Lagrangian mean curvature flow.
The Lagrangian mean curvature flow was devised mainly for searching minimal Lagrangian submanifolds in a Calabi-Yau manifold,
and becomes one of main tools in understanding Strominger-Yau-Zaslow's conjecture in Mirror symmetry \cite{SYZ}
and Thomas-Yau's conjecture \cite{ThomasYau}; See for instance \cite{ChenLi01,ChenLi02,Neves1,Neves2,Smoczyk3,Wang}.
For recent developments of the Lagrangian mean curvature flow in K\"ahler-Einstein manifolds,
see for instance survey papers \cite{Neves,Smoczyk2, WangMT} and the references therein.

A Lagrangian self-shrinker by definition is a self-shrinker satisfying the Lagrangian condition.
Type-I singularities along the Lagrangian mean curvature flow are modeled by self-shrinkers.
The study of Lagrangian self-shrinkers has draw some attentions recently.
In particular, many compact or noncompact examples are found, see for instance
\cite{Anciaux,CastroLerma,CastroLerma2,JoyceLeeTsui,LeeWang,LeeWang2}.

The F-functional and the entropy of a Lagrangian surface $\Sigma$ are given by (\ref{Ffunctional}) and (\ref{entropy}) respectively.
However if we stick to the Lagrangian mean curvature flow, admissible deformations are Lagrangian deformations,
i.e. those preserve the Lagrangian condition.
A normal vector field is called a Lagrangian variation if it is an infinitesimal Lagrangian deformation.
There is a correspondence between Lagrangian variations and closed $1$-forms via $X\leftrightarrow \theta:=-i_X\overline{\omega}$.
If $X$ is a normal vector field such that $-i_X\overline{\omega}$ is exact, we call $X$ a Hamiltonian variation.

When restricted to Lagrangian deformations, a critical point of $F_{x_0,t_0}$ is
a Lagrangian self-shrinker that becomes extinct at $(x_0,t_0)$
and a Lagrangian self-shrinker is a critical point of the entropy functional.
We call a Lagrangian self-shrinker $\Sigma$ Lagrangian entropy-stable
if it is a local minimum of the entropy under Lagrangian deformations,
and a Lagrangian self-shrinker $\Sigma$ is called Lagrangian (resp. Hamiltonian) F-stable
if for any $1$-parameter family of Lagrangian (resp. Hamiltonian) deformations $\Sigma_s$ of $\Sigma=\Sigma_0$,
there exist deformations $x_s$ of $x_0$ and $t_s$ of $t_0$ such that $(F_{x_s,t_s}(\Sigma_s))''\geq 0$ at $s=0$.

Using the correspondence between Lagrangian variations and closed $1$-forms,
we can rewrite the second variation of the F-functional in terms of closed $1$-form.
This correspondence was used by Oh \cite{Oh} in studying the Lagrangian stability of minimal Lagrangian submanifold in K\"ahelr manifold.
It was first observed by Smoczyk that any closed Lagrangian self-shrinker has nontrivial Maslov class, i.e. $[-i_H\overline{\omega}]\neq 0$,
hence its first Betti number $b_1\geq 1$, see \cite{CastroLerma2}.
In fact, the mean curvature form $-i_H\overline{\omega}$ is a twisted harmonic form (cf. Section 3).
We find that the Lagrangian variation associated to a twisted harmonic $1$-form $\theta\notin [-i_H\overline{\omega}]$
decreases the entropy. This allows us to prove the following

\begin{theorem}\label{thm1}
Any closed Lagrangian self-shrinker with $b_1\geq 2$ is Lagrangian $F$-unstable.
\end{theorem}

In $\mathbb{C}^2$, the embedded closed Lagrangian surface has very strict topological constraint. It has to be torus.
As a corollary, in $\mathbb{C}^2$ there are no embedded closed Lagrangian F-stable Lagrangian self-shrinkers.
In case that the closed Lagrangian self-shrinker has $b_1=1$, we are led to study the Hamiltonian F-stability.

\begin{theorem}\label{thm2}
For a Lagrangian self-shrinker $\Sigma$ with $b_1=1$, the Lagrangian F-stability of $\Sigma$ is equivalent to
the Hamiltonian F-stability of $\Sigma$.
\end{theorem}

However for $b_1\geq 2$, there is indeed a difference between Lagrangian F-stability and Hamiltonian F-stability.
For example the Clifford torus
$$T^n=\{(z^1,\cdots,z^n): \quad |z^1|^2=\cdots=|z^n|^2=2\}$$
is Hamiltonian F-stable but not Lagrangian F-stable.
It is an interesting question whether there exists a closed Lagrangian F-stable self-shrinker of dimension greater than one,
i.e. a closed Hamiltonian F-stable Lagrangian self-shrinker with $b_1=1$.
We also ask if there exists a complete noncompact Lagrangian F-stable Lagrangian self-shrinker which has polynomial volume growth,
besides Lagrangian planes and Lagrangian $S^1\times \mathbb{R}^{n-1}$?
It's also interesting to have more examples of Hamiltonian F-stable Lagrangian self-shrinkers.

We can give two characterizations of Hamiltonian F-stability for a closed Lagrangian self-shrinker with arbitrary $b_1(\geq 1)$
by its spectral property of the twisted Laplacian.
Without loss of generality, we assume that the closed Lagrangian self-shrinker becomes extinct at $(0,1)$, i.e. $H=-\frac{1}{2}x^\perp$.
The twisted Laplacian is then
$\triangle_f=\triangle-\frac{1}{2}<x^\top,\nabla \cdot>.$

\begin{theorem}\label{thm3}
A closed Lagrangian self-shrinker that becomes extinct at $(0,1)$ is Hamiltonian F-stable if and only if the twisted Laplacian $\triangle_f$ has
$$\lambda_1=\frac{1}{2}, \quad \Lambda_{\frac{1}{2}}=\{<x,w>, w\in \mathbb{R}^{2n}\}; \quad \lambda_2\geq 1.$$
\end{theorem}

\begin{theorem}\label{thm4}
A closed Lagrangian self-shrinker $\Sigma$ that becomes extinct at $(0,1)$ is Hamiltonian F-stable if and only if
\begin{equation*}\label{HamiltonianIneqn}
\int_\Sigma|du|^2e^{-\frac{|x|^2}{4}}d\mu\geq \int_\Sigma u^2e^{-\frac{|x|^2}{4}}d\mu, \quad \text{for all u s.t. }\quad
\int_\Sigma ue^{-\frac{|x|^2}{4}}d\mu=\int_\Sigma uxe^{-\frac{|x|^2}{4}}d\mu=0.
\end{equation*}
\end{theorem}

The classification problem of self-shrinkers with higher codimensions is much more complicated due to the complexity of the normal bundle,
see for instance \cite{AndrewsLiWei,Smoczyk5}.
Very recently, F-stability of self-shrinkers with higher codimensions was considered in \cite{AndrewsLiWei, ArezzoSun, LeeLue}.
The Lagrangian F-stability has also been considered by Lee and Lue \cite{LeeLue}.
In particular, they proved some of closed Lagrangian self-shrinkers in \cite{Anciaux} are Lagrangian F-unstable.

Colding and Minicozzi's idea of classifying self-similar solutions by employing entropy-stability and F-stability
also applies to other geometric flows,
for the harmonic map heat flow case see \cite{Zhang} and for the Yang-Mills flow case see \cite{ChenWangZhang}.
For Ricci shrinkers and Ricci-flat manifolds, an analogous stability to the F-stability is the linear stability.
A Ricci shrinker (resp. a Ricci-flat manifold) is called linearly stable if the second variation of Perelman's $\nu$-entropy (resp. $\lambda$-entropy) \cite{Perelman} is non-positive at the Ricci shrinker (resp. the Ricci-flat manifold), see \cite{CaoHamiltonIlmanen,CaoZhu}.
It was proved in \cite{DaiWangWei} that any compact Ricci-flat manifold admitting nontrivial parallel spinors is linearly stable.
For the special class of K\"ahler-Ricci solitons with Hodge number $h^{1,1}\geq 2$,
Hall and Murphy \cite{HallMurphy} proved that they are linearly unstable (allowing non-K\"ahlerian deformations),
which extended the result of Cao-Hamilton-Ilmanen \cite{CaoHamiltonIlmanen} in the K\"ahler-Einstein case.
In the contrast, if one considers deformations of K\"ahler metrics in the fixed class $c_1(M)$,
Tian and Zhu \cite{TianZhu} proved that the $\nu$-energy is maximized at a K\"ahler-Ricci soliton.

The paper is organized as follows: in the next section, we compute the first and second variation formula of the F-functional.
The variation formulas of the F-functional will be applied to
Lagrangian self-shrinkers and Lagrangian variations in Section 3,
where we study Lagrangian F-stability of closed Lagrangian self-shrinkers and prove Theorem \ref{thm1}.
In the last section we consider the Hamiltonian F-stability of closed Lagrangian self-shrinkers
and prove Theorem \ref{thm2}, \ref{thm3} and \ref{thm4}.

\section{F-functional, entropy, and second variational formula}\label{section2}

In this section we first recall Huisken's monotonicity formula,
which plays a crucial role in the formation of singularities along the mean curvature flow.
At a given singularity $(x,T)$, one can extract a tangent flow at $(x,T)$ from a sequence of rescaled flows.
Each time-slice of the tangent flow is a self-shrinker.
We also recall the F-functional and Colding-Minicozzi's entropy for immersed surfaces in Euclidean space.
The first and second variation formula of the F-functional will be calculated,
which is a slight modification of the hypersurface case \cite{CM}.
The calculations were also carried out in \cite{AndrewsLiWei,ArezzoSun,LeeLue}.
A critical point of the F-functional $F_{x_0,t_0}$ is a self-shrinkers that becomes extinct at $(x_0,t_0)$,
and the second variation formula gives rise to F-stability of self-shrinkers.
In the end of this section we give a characterization of the F-stability, see also \cite{AndrewsLiWei,ArezzoSun,LeeLue}.
The variation formulas of the F-functional will be applied to Lagrangian self-shrinkers in the next section.

Let $\Sigma$ be an $n$-dimensional complete immersed surface in $\mathbb{R}^N$, $g$ the induced metric on $\Sigma$,
and $\{e_i\}_{i=1}^n$ a local orthonormal frame of $T\Sigma$.
The second fundamental form and the mean curvature of $\Sigma$ are respectively given by
$$h_{ij}=(e_ie_j)^\perp, \quad H=(e_ie_i)^\perp.$$
here the superscript $\perp$ denotes the normal projection.
The projection to the tangential space will be denoted by the superscript $\top$.
If there exist positive constants $C_1, C_2$ and $d$ such that $Vol(B_r(0)\cap \Sigma)\leq C_1r^d+C_2$,
here $r$ is the Euclidean distance, we say that $\Sigma$ has polynomial volume growth.

Let $\Sigma_t$ be a family of immersed surfaces in $\mathbb{R}^N$, evolved by the mean curvature flow
\begin{equation}\label{meancurvatureflow}\nonumber
(\frac{\partial x}{\partial t})^\perp=H.
\end{equation}
Assume $T$ is the first singular time of the mean curvature flow.
For any $x_0\in \mathbb{R}^N, t_0>0$ and $t<\min\{t_0,T\}$, let
$$\rho_{x_0,t_0}(x,t)=[4\pi(t_0-t)]^{-\frac{n}{2}}e^{-\frac{|x-x_0|^2}{4(t_0-t)}}$$
and
$$\Phi_{x_0,t_0}(t)=\int_{\Sigma_t}\rho_{x_0,t_0}(x,t)d\mu_t.$$
Then Huisken's monotonicity formula \cite{Huisken} reads
$$\frac{d}{dt}\int_{\Sigma_t}\rho_{x_0,t_0}(x,t)d\mu_t
=-\int_{\Sigma_t}|H+\frac{(x-x_0)^\perp}{2(t_0-t)}|^2\rho_{x_0,t_0}(x,t)d\mu_t.$$

The monotonicity formula is crucial in understanding the formation of singularities along the mean curvature flow.
At a given singularity $(x,T)$ and for a given $\lambda_j>0$, one can consider the following rescaled flow
\begin{equation}\label{blowup}
\widetilde{\Sigma}_s^{\lambda_j}:=\lambda_j(\Sigma_{T+\lambda_j^{-2}s}-x), \quad s<0.
\end{equation}
Huisken \cite{Huisken} showed that if the mean curvature flow develops a Type-I singularity at time $T$, i.e.
$(T-t)\sup_{\Sigma_t}|h_{ij}|^2$ is uniformly bounded in $t$,
there exists a sequence $\lambda_j\rightarrow \infty$ such that the sequence $\widetilde{\Sigma}_s^{\lambda_j}$
converges smoothly to a limiting flow $\widetilde{\Sigma}_s$.
In general using the monotonicity formula and Brakke's compactness theorem \cite{Brakke}, Ilmanen \cite{Ilmanen} and White \cite{White}
proved that at any given singularity $(x,T)$,
there exists a sequence $\lambda_j\rightarrow \infty$ such that the sequence $\widetilde{\Sigma}_s^{\lambda_j}$
converges weakly to a limiting flow $\widetilde{\Sigma}_s$.
$\widetilde{\Sigma}_s$ is called a tangent flow at $(x,T)$.
Moreover for each $s<0$, $\widetilde{\Sigma}_s$ is a self-shrinker.

If the mean curvature flow is initiating from a compact immersed surface,
Colding and Minicozzi \cite{CM} showed that each time-slice of $\widetilde{\Sigma}_s$ has polynomial volume growth.
From now on we restrict ourselves to self-shrinkers which are smooth and have polynomial volume growth.
More specifical than (\ref{selfshrinker1}), we call an immersed surface in $\mathbb{R}^N$ a self-shrinker that becomes extinct at $(x_0,t_0)$
if it satisfies
\begin{equation}\label{selfshrinker3}
H+\frac{(x-x_0)^\perp}{2t_0}=0.
\end{equation}
For a self-shrinker that becomes extinct at $(x_0,t_0)$, $\sqrt{t_0-t}(\Sigma-x_0)$ defines a homothetically shrinking mean curvature flow.
Note that a self-shrinker that becomes extinct at $(x_0, t_0)$ is also a steady point of Huisken's monotonicity quantity $\Phi_{x_0,t_0}$.

Let $\Sigma$ be a complete self-shrinker which becomes extinct at $(x_0,t_0)$ and has polynomial volume growth.
Colding and Minicozzi \cite{CM} introduced an operator acting on functions on $\Sigma$ by
\begin{equation}\label{twistedLaplacian}
\triangle u-\frac{1}{2t_0}<(x-x_0)^\top,\nabla u>=e^{\frac{|x-x_0|^2}{4t_0}}\text{div}(e^{-\frac{|x-x_0|^2}{4t_0}}\nabla u).
\end{equation}
Analogous operator also appear in other backgrounds, see for instance \cite{Futaki,Lott}.
Given a function $f$, the so-called twisted Laplacian is defined by
$$\triangle_fu=\triangle u-g(\nabla f,\nabla u).$$
The twisted Laplacian on functions is actually $-d_f^*d$,
here $d_f^*$ is the adjoint operator of $d$ with respect to the measure $e^{-f}d\mu$, see for instance \cite{Futaki}.
On the self-shrinker $\Sigma$ which becomes extinct at $(x_0,t_0)$, we choose  $f=\frac{|x-x_0|^2}{4t_0}$. Then
\begin{equation}\label{twistedLaplacian2}
\triangle_f=\triangle -\frac{1}{2t_0}<(x-x_0)^\top,\nabla \cdot>.
\end{equation}
The following "weighted $W^{2,2}$ space" $W_w^{2,2}$  was also introduced in \cite{CM}
$$W_w^{2,2}=\{u| \int_\Sigma(|u|^2+|\nabla u|^2+|\triangle_fu|^2)e^{-\frac{|x-x_0|^2}{4t_0}}d\mu<\infty \}.$$
For any $u, v\in W_w^{2,2}$, it holds that
\begin{equation}\label{integrationbyparts}
\int_\Sigma u\triangle_fve^{-\frac{|x-x_0|^2}{4t_0}}d\mu=-\int_\Sigma g(\nabla u,\nabla v)e^{-\frac{|x-x_0|^2}{4t_0}}d\mu=\int_\Sigma v\triangle_fue^{-\frac{|x-x_0|^2}{4t_0}}d\mu.
\end{equation}

\begin{lemma}\label{Lemma}
Let $\Sigma$ be an $n$-dimensional complete self-shrinker which becomes extinct at $(x_0,t_0)$ and has polynomial volume growth,
$w$ a vector in $\mathbb{R}^{N}$.
Then
\begin{itemize}
\item[(1)] $\int_\Sigma<x-x_0,w>e^{-\frac{|x-x_0|^2}{4t_0}}d\mu=0$;
\item[(2)] $\int_\Sigma(|x-x_0|^2-2nt_0)e^{-\frac{|x-x_0|^2}{4t_0}}d\mu=0$;
\item[(3)] $\int_\Sigma[|x-x_0|^4-4n(n+2)t_0^2+16t_0^3|H|^2]e^{-\frac{|x-x_0|^2}{4t_0}}d\mu=0$;
\item[(4)] $\int_\Sigma[<x-x_0,w>^2-2t_0|w^\top|^2]e^{-\frac{|x-x_0|^2}{4t_0}}d\mu=0$;
\item[(5)] $\int_\Sigma |x-x_0|^2<x-x_0,w>e^{-\frac{|x-x_0|^2}{4t_0}}d\mu=0$;
\item[(6)] $\int_\Sigma <(x-x_0)^\top,w>e^{-\frac{|x-x_0|^2}{4t_0}}d\mu=0$.
\end{itemize}
\end{lemma}

\begin{proof}
The proof is similar to the hypersurface case in \cite{CM}.

(1) Let
$$u=<x-x_0,w>, \quad v=1.$$
By (\ref{twistedLaplacian2}),
\begin{eqnarray*}
\triangle_f u&=& <H,w>-\frac{1}{2t_0}<(x-x_0)^\top,w>
\\&=&-\frac{1}{2t_0}<x-x_0,w>.
\end{eqnarray*}
Therefore, (1) follows from (\ref{integrationbyparts}).

(2) Let
$$u=|x-x_0|^2, \quad v=1.$$
The identity then follows from (\ref{integrationbyparts}) and
\begin{eqnarray*}
\triangle_f u&=&2n+2<x-x_0,H>-\frac{1}{t_0}|(x-x_0)^\top|^2
\\&=&2n-\frac{1}{t_0}|x-x_0|^2.
\end{eqnarray*}

(3) Let
$$u=v=|x-x_0|^2.$$
Then
$$\triangle_f u=2n-\frac{1}{t_0}|x-x_0|^2, \quad |\nabla u|^2=4|(x-x_0)^\top|^2.$$
Hence it follows from (\ref{integrationbyparts}) that
\begin{eqnarray*}
&&\int_\Sigma (2n-\frac{1}{t_0}|x-x_0|^2)|x-x_0|^2e^{-\frac{|x-x_0|^2}{4t_0}}d\mu
\\&=&-\int_\Sigma 4|(x-x_0)^\top|^2e^{-\frac{|x-x_0|^2}{4t_0}}d\mu
\\&=&-\int_\Sigma 4(|x-x_0|^2-4t_0^2|H|^2)e^{-\frac{|x-x_0|^2}{4t_0}}d\mu.
\end{eqnarray*}
By using (2), we get
\begin{eqnarray*}
&&\int_\Sigma |x-x_0|^4e^{-\frac{|x-x_0|^2}{4t_0}}d\mu
\\&=&\int_\Sigma ((2n+4)t_0|x-x_0|^2-16t_0^3|H|^2)e^{-\frac{|x-x_0|^2}{4t_0}}d\mu
\\&=&\int_\Sigma (4n(n+2)t_0^2-16t_0^3|H|^2)e^{-\frac{|x-x_0|^2}{4t_0}}d\mu.
\end{eqnarray*}

(4) Let
$$u=v=<x-x_0,w>.$$
We have
$$\triangle_f u=-\frac{1}{2t_0}<x-x_0, w>, \quad \nabla u=w^\top.$$
Hence
\begin{eqnarray*}
\int_\Sigma -\frac{1}{2t_0}<x-x_0,w>^2e^{-\frac{|x-x_0|^2}{4t_0}}d\mu
=-\int_\Sigma |w^\top|^2e^{-\frac{|x-x_0|^2}{4t_0}}d\mu.
\end{eqnarray*}

(5, 6)  Let
$$u=<x-x_0,w>, \quad v=|x-x_0|^2.$$
We have
$$\triangle_f u=-\frac{1}{2t_0}<x-x_0, w>, \quad \nabla u=w^\top$$
and
$$\triangle_f v=2n-\frac{1}{t_0}|x-x_0|^2,\quad   \nabla v=2(x-x_0)^\top.$$
Hence by (\ref{integrationbyparts}), we get
\begin{eqnarray*}
&&\int_\Sigma <x-x_0,w>(2n-\frac{1}{t_0}|x-x_0|^2)e^{-\frac{|x-x_0|^2}{4t_0}}d\mu
\\&=&\int_\Sigma -<w^\top,2(x-x_0)^\top>e^{-\frac{|x-x_0|^2}{4t_0}}d\mu
\\&=&\int_\Sigma -\frac{1}{2t_0}<x-x_0,w>|x-x_0|^2e^{-\frac{|x-x_0|^2}{4t_0}}d\mu.
\end{eqnarray*}
Then it follows from (1) that
\begin{eqnarray*}
&&\int_\Sigma -\frac{1}{t_0}<x-x_0,w>|x-x_0|^2e^{-\frac{|x-x_0|^2}{4t_0}}d\mu
\\&=&\int_\Sigma -\frac{1}{2t_0}<x-x_0,w>|x-x_0|^2e^{-\frac{|x-x_0|^2}{4t_0}}d\mu.
\end{eqnarray*}
Hence,
$$\int_\Sigma |x-x_0|^2<x-x_0,w>e^{-\frac{|x-x_0|^2}{4t_0}}d\mu=\int_\Sigma <(x-x_0)^\top,w>e^{-\frac{|x-x_0|^2}{4t_0}}d\mu=0.$$
\end{proof}

For an immersed $n$-dimensional surface $\Sigma$ in $\mathbb{R}^N$, the F-functional (with respect to $x_0\in \mathbb{R}^N, t_0>0$)
and the entropy \cite{CM} are respectively defined by
$$F_{x_0,t_0}(\Sigma)=\int_\Sigma (4\pi t_0)^{-\frac{n}{2}}e^{-\frac{|x-x_0|^2}{4t_0}}d\mu$$
and
\begin{equation*}
\lambda(\Sigma)=\sup_{x_0,t_0}F_{x_0,t_0}(\Sigma).
\end{equation*}
The relation between the F-functional and Huisken's monotonicity quantity is given by
\begin{equation}\label{FPhi}
F_{x_0,t_0}(\Sigma_{t_1})=\Phi_{x_0,t_0+t_1}(\Sigma_{t_1}),
\end{equation}
here $\Sigma_t$ is a $1$-parameter family of immersions.  The entropy functional has very nice properties, for example
(i) $\lambda$ is invariant under dilations and rigid motions; (ii) $\lambda$ is non-increasing along the mean curvature flow.
The property (ii) follows from (\ref{FPhi}) and Huisken's monotonicity formula.
One easily sees that the surface $\widetilde{\Sigma}_s^{\lambda_j}$ in (\ref{blowup}) satisfies
$\lambda(\widetilde{\Sigma}_s^{\lambda_j})\leq \lambda(\Sigma_0)$.
We now compute the first variation formula of the F-functional.

\begin{proposition}
Let $\Sigma$ be a complete immersed surface in $\mathbb{R}^N$ and $\Sigma_s$ a family of deformations of $\Sigma$
generated by normal variation $X_s$, $s\in (-\epsilon,\epsilon)$.
Let $x_s, t_s$ be deformations of $x_0$ and $t_0$ respectively with velocities $\dot{x}_s$ and $\dot{t}_s$,
then $\frac{d}{d s}F_{x_s,t_s}(\Sigma_s)$ is given by
\begin{equation*}\label{firstvariation1}
\int_{\Sigma_s}[-<X_s,H+\frac{x-x_s}{2t_s}>+\frac{1}{2t_s}<x-x_s,\dot{x}_s>
+\dot{t}_s(\frac{|x-x_s|^2}{4t_s^2}-\frac{n}{2t_s})](4\pi t_s)^{-\frac{n}{2}}e^{-\frac{|x-x_s|^2}{4t_s}}d\mu_s.
\end{equation*}
In particular if $X_s|_{s=0}=X, \dot{x}_0=y, \dot{t}_0=h$, then $\frac{d}{d s}|_{s=0}F_{x_s,t_s}(\Sigma_s)$ is given by
\begin{equation*}\label{firstvariation2}
\int_{\Sigma}[-<X,H+\frac{x-x_0}{2t_0}>+\frac{1}{2t_0}<x-x_0,y>
+h(\frac{|x-x_0|^2}{4t_0^2}-\frac{n}{2t_0})] (4\pi t_0)^{-\frac{n}{2}}e^{-\frac{|x-x_0|^2}{4t_0}}d\mu.
\end{equation*}
\end{proposition}

\begin{proof}
By definition, $F_{x_s,t_s}(\Sigma_s)=\int_{\Sigma_s}(4\pi t_s)^{-\frac{n}{2}}e^{-\frac{|x-x_s|^2}{4t_s}}d\mu_s$.
The first variation formula of the F-functional then follows from
$$\frac{d}{d s}d\mu_s=-<X_s,H>d\mu_s$$
$$\frac{d}{d s}[(4\pi t_s)^{-\frac{n}{2}}e^{-\frac{|x-x_s|^2}{4t_s}}]
=(-\frac{n}{2t_s}\dot{t}_s+\frac{|x-x_s|^2}{4t_s^2}\dot{t}_s-\frac{1}{2t_s}<x-x_s,X_s-\dot{x}_s>)(4\pi t_s)^{-\frac{n}{2}}e^{-\frac{|x-x_s|^2}{4t_s}}.$$
\end{proof}

From the first variation formula, we see a critical point of $F_{x_0,t_0}$ is a self-shrinker that becomes extinct at $(x_0,t_0)$.
Moreover if $\Sigma$ is a self-shrinker that becomes extinct at $(x_0,t_0)$,
then by Lemma \ref{Lemma} one sees that $(\Sigma,x_0,t_0)$ is a critical point of the F-functional,
which also means that $\Sigma$ is a critical point of the entropy.
We now compute the second variation of the $F$-functional at a self-shrinker.
For a normal vector field $X$, denote
\begin{equation}\label{Jacobioperator}
LX=\triangle X+<X,h_{ij}>h_{ij}-<\frac{x-x_0}{2t_0},e_i>\nabla_{e_i}X+\frac{1}{2t_0}X.
\end{equation}

\begin{theorem}\label{secondvariation}
Let $\Sigma$ be a self-shrinker which becomes extinct at $(x_0,t_0)$ and has polynomial volume growth.
Assume $\Sigma_s, x_s$ and $t_s$ are deformations of $\Sigma, x_0$ and $t_0$ respectively with
$$\frac{\partial \Sigma_s}{\partial s}|_{s=0}=X, \quad \frac{\partial x_s}{\partial s}|_{s=0}=y, \quad \frac{\partial t_s}{\partial s}|_{s=0}=h.$$
Then
\begin{eqnarray*}
\frac{d^2}{d^2s}|_{s=0}F_{x_s,t_s}(\Sigma_s)&=&\int_{\Sigma}-<X,L X-\frac{1}{t_0}y+\frac{2}{t_0}hH>(4\pi t_0)^{-\frac{n}{2}}e^{-\frac{|x-x_0|^2}{4t_0}}d\mu
\\&&-\int_{\Sigma}[\frac{1}{2t_0}|y^\perp|^2+\frac{1}{t_0}h^2|H|^2] (4\pi t_0)^{-\frac{n}{2}}e^{-\frac{|x-x_0|^2}{4t_0}}d\mu.
\end{eqnarray*}
\end{theorem}

\begin{proof}
For convenience, we write
$$F''=\frac{d^2}{d^2s}|_{s=0}F_{x_s,t_s}(\Sigma_s), \quad G_s(x)=(4\pi t_s)^{-\frac{n}{2}}e^{-\frac{|x-x_s|^2}{4t_s}}, \quad G(x)=(4\pi t_0)^{-\frac{n}{2}}e^{-\frac{|x-x_0|^2}{4t_0}}.$$
It follows from the first variation formula of the F-functional that $F''$ is
\begin{eqnarray*}
\frac{d}{d s}|_{s=0}\int_{\Sigma_s}[-<X_s,H+\frac{x-x_s}{2t_s}>+\frac{1}{2t_s}<x-x_s,\dot{x}_s>
+\dot{t}_s(\frac{|x-x_s|^2}{4t_s^2}-\frac{n}{2t_s})]G_sd\mu_s.
\end{eqnarray*}
Under the normal deformation with $\frac{\partial \Sigma_s}{\partial s}=X$, we have
$$\frac{\partial}{\partial s}g_{ij}=-2<X,h_{ij}>$$
and
\begin{eqnarray*}
(\frac{\partial}{\partial s}H)^\perp=\triangle X+<X,h_{ij}>h_{ij}.
\end{eqnarray*}
Then it follows from Lemma \ref{Lemma} that
\begin{eqnarray*}
F''&=&\int_{\Sigma}-[<\frac{\partial X_s}{\partial s}|_{s=0},\frac{(x-x_0)^\top}{2t_0}>+<X,\triangle X>+<X,h_{ij}>^2]Gd\mu
\\&&+\int_{\Sigma}[-<X,\frac{X-y}{2t_0}>+<X,\frac{h(x-x_0)}{2t_0^2}>+\frac{1}{2t_0}<X-y,y>]Gd\mu
\\&&+\int_{\Sigma}h(-\frac{|x-x_0|^2h}{2t_0^3}+\frac{<x-x_0,X-y>}{2t_0^2}+\frac{nh}{2t_0^2})Gd\mu
\\&&+\int_{\Sigma}[\frac{1}{2t_0}<x-x_0,y>+h(\frac{|x-x_0|^2}{4t_0^2}-\frac{n}{2t_0})]^2Gd\mu.
\end{eqnarray*}
Note that
$$<\frac{\partial X}{\partial s}|_{s=0},\frac{(x-x_0)^\top}{2t_0}>=-<\frac{x-x_0}{2t_0},e_i><X,\nabla_{e_i}X>.$$
We denote
$$<\frac{x-x_0}{2t_0},e_i>\nabla_{e_i}X=\frac{x-x_0}{2t_0}\cdot\nabla X.$$
By the assumption on $\Sigma$ and Lemma \ref{Lemma},
\begin{eqnarray*}
F''&=&\int_{\Sigma}-[-<X,\frac{x-x_0}{2t_0}\cdot\nabla X>+<X,\triangle X>+<X,h_{ij}>^2]Gd\mu
\\&&+\int_{\Sigma}[-\frac{1}{2t_0}|X|^2+\frac{1}{t_0}<X,y>-\frac{1}{2t_0}|y|^2-\frac{2h}{t_0}<X,H>-\frac{nh^2}{2t_0^2}] Gd\mu
\\&&+\int_{\Sigma}[\frac{1}{2t_0}<x-x_0,y>+h(\frac{|x-x_0|^2}{4t_0^2}-\frac{n}{2t_0})]^2Gd\mu.
\end{eqnarray*}
By Lemma \ref{Lemma} again,
\begin{eqnarray*}
&&\int_{\Sigma}[\frac{1}{2t_0}<x-x_0,y>+h(\frac{|x-x_0|^2}{4t_0^2}-\frac{n}{2t_0})]^2Gd\mu
\\&=&\int_{\Sigma}[\frac{1}{4t_0^2}<x-x_0,y>^2+h^2(\frac{|x-x_0|^2}{4t_0^2}-\frac{n}{2t_0})^2]Gd\mu
\\&=&\int_{\Sigma}[\frac{1}{2t_0}|y^\top|^2+h^2(\frac{n}{2t_0^2}-\frac{1}{t_0}|H|^2)]Gd\mu.
\end{eqnarray*}
Hence we get
\begin{eqnarray*}
F''&=&\int_{\Sigma}-<X,\triangle X+<X,h_{ij}>h_{ij}-\frac{x-x_0}{2t_0}\cdot\nabla X+\frac{1}{2t_0}X> Gd\mu
\\&&+\int_{\Sigma}<X,\frac{1}{t_0}y-\frac{2}{t_0}h H> Gd\mu-\int_{\Sigma}[\frac{1}{2t_0}|y^\perp|^2+\frac{1}{t_0}h^2|H|^2)]Gd\mu.
\end{eqnarray*}
\end{proof}

Without loss of generality, from now on we assume $x_0=0, t_0=1$.
In particular a self-shrinker that becomes extinct at $(0,1)$ satisfies $H=-\frac{1}{2}x^\perp$ and the Jacobi operator $L$ becomes
$$LX=\triangle X+<X,h_{ij}>h_{ij}-<\frac{x}{2},e_i>\nabla_{e_i}X+\frac{1}{2}X.$$
Let $L_f^2$ be the Hilbert space consisting of all functions square integrable with respect to the measure $e^{-\frac{|x|^2}{4}}d\mu$.
When referring to a normal variation field $X$, we assume $|X|, |\nabla X|$ and $|LX|$ are in $L_f^2$.
By the second variation formula of the F-functional, we arrive at the following definition of F-stability \cite{CM}.

\begin{definition}\label{Fstability}
A self-shrinker $\Sigma$ which becomes extinct at $(0,1)$ and has polynomial volume growth
is called F-stable if for any normal variation $X$,
there exist a vector $y$ and a constant $h$ such that
\begin{eqnarray*}
F''=(4\pi)^{-\frac{n}{2}}\int_{\Sigma}[-<X,L X>+<X,y>-2h<X,H>-\frac{1}{2}|y^\perp|^2-h^2|H|^2]e^{-\frac{|x|^2}{4}}d\mu\geq 0.
\end{eqnarray*}
\end{definition}

A normal vector field $X$ is called an eigenvector field of $L$ and of eigenvalue $\mu$ if there exists a constant $\mu$ such that $LX=-\mu X$.
Note that $L$ is self-adjoint with respect to the measure $e^{-\frac{|x|^2}{4}}d\mu$.
We now show that $H$ and $w^\perp$ are eigenvector fields of $L$, here $w$ is any vector in $\mathbb{R}^N$.
We also show that a self-shrinker is F-stable if and only if $H$ and $w^\perp$ are the only eigenvector fields which have negative eigenvalues.
Note that for a self-shrinker the variations $H$ and $w^\perp$ correspond to dilation and translation respectively.

\begin{lemma}\label{eigenvectorfields}
Let $\Sigma$ be a self-shrinker that becomes extinct at $(0,1)$ and $w$ a vector in $\mathbb{R}^N$. Then
$$LH=H,\quad Lw^\perp=\frac{1}{2}w^\perp.$$
\end{lemma}

\begin{proof}
Let $\nu_\alpha$ be a local orthonormal frame on the normal bundle of $\Sigma$ and normal at the point under consideration.

Write $H=h_{kk\alpha}\nu_\alpha=-\frac{1}{2}<x,\nu_\alpha>\nu_\alpha.$
Then
$$\nabla_iH=\frac{1}{2}h_{il\alpha}<x,e_l>\nu_\alpha=\frac{1}{2}<x,e_l>h_{il}$$
and
\begin{eqnarray*}
\triangle H&=&\frac{1}{2}(H+<x,h_{il}>h_{il}+<x,e_l>\nabla_lH)
\\&=&\frac{1}{2}H-<H,h_{il}>h_{il}+\frac{1}{2}<x,e_l>\nabla_lH.
\end{eqnarray*}
Hence we get
$$L H=\triangle H+<H,h_{ij}>h_{ij}-<\frac{x}{2},e_i>\nabla_iH+\frac{1}{2}H=H.$$

Write $w^\perp=<w,\nu_\alpha>\nu_\alpha$. Then
$$\nabla_iw^\perp=-h_{il\alpha}<w,e_l>\nu_\alpha=-<w,e_l>h_{il},$$
and
\begin{eqnarray*}
\triangle w^\perp&=&-<w,h_{il}>h_{il}-<w,e_l>\nabla_l H
\\&=&-<w,h_{il}>h_{il}-\frac{1}{2}<w,e_l><x,e_k>h_{lk}
\\&=&-<w,h_{il}>h_{il}+\frac{1}{2}<x,e_k>\nabla_kw^\perp.
\end{eqnarray*}
Hence we get $Lw^\perp=\frac{1}{2}w^\perp$.
\end{proof}

\begin{theorem}\label{characterization}
A self-shrinker $\Sigma$ is $F$-stable if and only if $H$ and $w^\perp$ are the only eigenvector fields of $L$
which have negative eigenvalues.
\end{theorem}
\begin{proof}
Assume $\Sigma$ is Lagrangian F-stable and $X$ is an eigenvector field with $LX=-\mu X$ such that
$$\int_\Sigma <X,H>e^{-\frac{|x|^2}{4}}d\mu=\int_\Sigma <X,w^\perp>e^{-\frac{|x|^2}{4}}d\mu=0, \quad \forall w\in\mathbb{R}^N.$$
Then
$$F''=(4\pi)^{-\frac{n}{2}}\int_{\Sigma}[\mu |X|^2-\frac{1}{2}|y^\perp|^2-h^2|H|^2]e^{-\frac{|x|^2}{4}}d\mu.$$
Therefore the F-stability of $\Sigma$ implies that $\mu\geq 0$.

On the other hand, any normal variation $X$ admit a unique decomposition $$X=a_0H+w^\perp+X_1$$ such that
$X_1$ is orthogonal to $H$ and all translations $y^\perp$.
Then by Lemma \ref{eigenvectorfields},
\begin{eqnarray*}
F''&=&(4\pi)^{-\frac{n}{2}}\int_{\Sigma}-<a_0H+w^\perp+X_1,a_0H+\frac{1}{2}w^\perp+L X_1>e^{-\frac{|x|^2}{4}}d\mu
\\&&+(4\pi)^{-\frac{n}{2}}\int_{\Sigma}[<a_0H+w^\perp+X_1,y>-2h<a_0H+w^\perp+X_1,H>]e^{-\frac{|x|^2}{4}}d\mu
\\&&-(4\pi)^{-\frac{n}{2}}\int_{\Sigma}(\frac{1}{2}|y^\perp|^2+h^2|H|^2)e^{-\frac{|x|^2}{4}}d\mu
\\&=&(4\pi)^{-\frac{n}{2}}\int_{\Sigma}[-<X_1,LX_1>-(a_0+h)^2|H|^2-\frac{1}{2}|y^\perp-w^\perp|^2]e^{-\frac{|x|^2}{4}}d\mu.
\end{eqnarray*}
Taking $h=-a_0$ and $y=w$, we are done.
\end{proof}

The characterization of F-stability was proved by Lee and Lue \cite{LeeLue}.
We get the characterization independently. The necessary part of the characterization was also proved in \cite{AndrewsLiWei}.
An interesting application of the characterization was found in \cite{LeeLue}, which states that
the product of two self-shrinkers which become extinct at the same $(x_0,t_0)$ is F-unstable.

A similar characterization for the linear stability of Ricci shrinkers was proved by Cao and Zhu \cite{CaoZhu},
and a similar characterization of the F-stability for self-similar solutions to the harmonic map heat flow was obtained in \cite{Zhang}.

\section{Lagrangian F-stability of closed Lagrangian self-shrinkers}

In this section we study Lagrangian F-stability of closed Lagrangian self-shrinkers.
Our starting point is a correspondence between Lagrangian variations and closed $1$-forms for Lagrangian surfaces,
which was also used by Oh \cite{Oh} in the studying of
Lagrangian stability of minimal Lagrangian submanifolds in K\"ahler manifolds.
Via this correspondence, we replace a Lagrangian variation by closed a $1$-form in the second variation formula of the F-functional.
Twisted harmonic $1$-form will play a crucial role in the proof of Theorem \ref{thm1}.
In particular, we will prove that any nontrivial twisted harmonic $1$-form in a different class from the Maslov class is an obstruction to the F-stability.

Let $(\mathbb{C}^n,\overline{g},J,\overline{\omega})$ be the complex Euclidean space with
the standard metric, complex structure and K\"ahler form such that
$\overline{g}=\overline{\omega}(\cdot,J\cdot)$.
An $n$-dimensional immersed surface $\Sigma$ is called a Lagrangian submanifold if $\overline{\omega}|_{T\Sigma}=0$.
Let $\Sigma$ be a Lagrangian surface in $\mathbb{C}^n$,
$\{e_i\}_{i=1}^n$ a local orthonormal frame of $T\Sigma$ and $\nu_i=Je_i$.
Note $\{\nu_i\}_{i=1}^n$ also form a local orthonormal frame of the normal bundle.
The second fundamental form is defined by
$$h_{ijk}=\overline{g}(\overline{\nabla}_{e_i}e_j,\nu_k),$$
which is symmetric in $i,j$ and $k$. The mean curvature vector field $H$ is defined by
$$H=H_k\nu_k=(\overline{\nabla}_{e_i}e_i)^\perp=h_{iik}\nu_k.$$

Let $\{e^i\}_{i=1}^n$ be the dual basis of $\{e_i\}_{i=1}^n$.
For any normal vector field $X=X_k\nu_k$, one can associated to it a $1$-form on $\Sigma$ by
$$\theta=-i_X\overline{\omega}=X_ke^k.$$
A normal vector field $X$ is called a Lagrangian variation if $i_X\overline{\omega}$ is closed.
It is easy to check that Lagrangian variations are infinitesimal Lagrangian deformations.
A normal variation field $X$ is called a Hamiltonian variation if $i_X\overline{\omega}$ is exact.
For example, the mean curvature vector field $H$ of a Lagrangian submanifold in $\mathbb{C}^n$ is a Lagrangian variation.
The corresponding closed $1$-form $-i_H\overline{\omega}$ is called the mean curvature form, still denoted by $H$, and the cohomology class
$[-i_H\overline{\omega}]$ is called the Maslov class.

For Lagrangian surfaces, the F-functional and the entropy are the same as the definitions in the last section.
The first variation formula of the F-functional in last section still holds
but $\Sigma_s$ is now replaced  by a Lagrangian deformation and $X_s$ is replaced by a Lagrangian variation.
Note that $H$ and $(x-x_0)^\perp$ are both Lagrangian variations.
Therefore a critical point of $F_{x_0,t_0}$ is a Lagrangian self-shrinker even when we are restricted to Lagrangian deformations.
Similarly by the second variation formula of the F-functional, we have the following definition of Lagrangian (resp. Hamiltonian) F-stability
for Lagrangian self-shrinkers.

\begin{definition}
Let $\Sigma$ be a complete Lagrangian self-shrinker that becomes extinct at $(0,1)$.
$\Sigma$ is called Lagrangian (resp. Hamiltonian) $F$-stable if for any Lagrangian (resp. Hamiltonian) variation $X$,
there exist a vector $y$ and a constant $h$ such that
\begin{eqnarray*}
F''=(4\pi)^{-\frac{n}{2}}\int_{\Sigma}[-<X,L X>+<X,y>-2h<X,H>-\frac{1}{2}|y^\perp|^2-h^2|H|^2]e^{-\frac{|x|^2}{4}}d\mu\geq 0.
\end{eqnarray*}
\end{definition}

We now rewrite the second variation in terms of the closed $1$-form $\theta=-i_X\overline{\omega}=X_ke^k$.
We first decompose the closed $1$-form into its harmonic part and exact part by
\begin{equation}\label{Hodgedecomposition}
\theta=\theta_0+du.
\end{equation}
Let $d^*$ be the adjoint operator of $d$ and $\triangle_H=d^*d+dd^*$ be the Hodge Laplacian.

\begin{proposition}\label{secondvariationofLagrangian}
Let $\Sigma$ be a closed Lagrangian self-shrinker that becomes extinct at $(0,1)$ and $X$ a Lagrangian variation.
Let $\theta=-i_X\overline{\omega}$ and $\theta=\theta_0+du$ be the decomposition (\ref{Hodgedecomposition}).
Then the second variation is given by
\begin{eqnarray}\label{secondvariationformulaLag}
F''&=&(4\pi)^{-\frac{n}{2}}\int_{\Sigma}[|d^*du+\frac{1}{2}du(x^\top)+\frac{1}{2}\theta_0(x^\top)|^2-|\theta|^2]e^{-\frac{|x|^2}{4}}d\mu\nonumber
\\&&+(4\pi)^{-\frac{n}{2}}\int_{\Sigma}[-\theta(Jy^\perp)-h\theta(Jx^\perp)-\frac{1}{2}|y^\perp|^2-\frac{1}{4}h^2|x^\perp|^2]e^{-\frac{|x|^2}{4}}d\mu.
\end{eqnarray}
\end{proposition}

\begin{proof}
For $X=X_k\nu_k$, we have
$$LX=\triangle X_k\nu_k+X_ph_{ijp}h_{ijk}\nu_k-<\frac{x}{2},e_i>\nabla_iX_k\nu_k+\frac{1}{2}X_k\nu_k$$
and
\begin{eqnarray*}
F''&=&(4\pi)^{-\frac{n}{2}}\int_{\Sigma}[-X_k\triangle X_k-X_kX_ph_{ijp}h_{ijk}+X_k<\frac{x}{2},e_i>\nabla_iX_k-\frac{1}{2}|X|^2]e^{-\frac{|x|^2}{4}}d\mu
\\&&+(4\pi)^{-\frac{n}{2}}\int_{\Sigma}[<X,y>-2h<X,H>-\frac{1}{2}|y^\perp|^2-h^2|H|^2]e^{-\frac{|x|^2}{4}}d\mu.
\end{eqnarray*}
Note that
$$(dd^*\theta)(e_k)=-\nabla_{e_k}\nabla_{e_j}\theta_j,$$
$$(d^*d\theta)(e_k)=-\nabla_{e_j}\nabla_{e_j}\theta_k+\nabla_{e_j}\nabla_{e_k}\theta_j,$$
then the Hodge Laplacian of $\theta$ is given by
\begin{eqnarray*}
(\triangle_H\theta)(e_k)&=&-\nabla_{e_k}\nabla_{e_j}\theta_j-\nabla_{e_j}\nabla_{e_j}\theta_k+\nabla_{e_j}\nabla_{e_k}\theta_j
\\&=&-\nabla_{e_j}\nabla_{e_j}\theta_k+R_{kl}\theta_l.
\end{eqnarray*}
In $\mathbb{C}^n$, the Ricci curvature of $\Sigma$ is given by
$$R_{kl}=<h_{kl},H>-<h_{kp},h_{lp}>.$$
Hence
\begin{eqnarray*}
(\triangle_H\theta)(e_k)&=&-\nabla_{e_j}\nabla_{e_j}\theta_k+(<h_{kl},H>-<h_{kp},h_{lp}>)\theta_l
\end{eqnarray*}
and
$$<\theta,\triangle_H\theta>=-X_k\triangle X_k-\frac{1}{2}<x,h_{kl}>X_kX_l-X_kX_lh_{kpq}h_{lpq}.$$
Then
\begin{eqnarray*}
F''&=&(4\pi)^{-\frac{n}{2}}\int_{\Sigma}[<\theta,\triangle_H\theta>+\frac{1}{2}<x,h_{kl}>X_kX_l+X_k<\frac{x}{2},e_i>\nabla_iX_k]
e^{-\frac{|x|^2}{4}}d\mu
\\&&+(4\pi)^{-\frac{n}{2}}\int_{\Sigma}[-\frac{1}{2}|\theta|^2+<X,y>-2h<X,H>-\frac{1}{2}|y^\perp|^2-h^2|H|^2]e^{-\frac{|x|^2}{4}}d\mu.
\end{eqnarray*}
We calculate the second term by
\begin{eqnarray*}
&&\int_{\Sigma}\frac{1}{2}<x,h_{kl}>X_kX_le^{-\frac{|x|^2}{4}}d\mu
\\&=&\int_{\Sigma}\frac{1}{2}[e_k(<x,e_l>X_l)-X_k-<x,e_l>\nabla_kX_l]X_ke^{-\frac{|x|^2}{4}}d\mu
\\&=&\int_{\Sigma}[\frac{1}{2}\theta(x^\top)d^*\theta+\frac{1}{4}|\theta(x^\top)|^2-\frac{1}{2}|\theta|^2-\frac{1}{2}X_k<x,e_i>\nabla_kX_i ]e^{-\frac{|x|^2}{4}}d\mu.
\end{eqnarray*}
Note that $X$ is Lagrangian, i.e. $\nabla_iX_k=\nabla_kX_i$. Therefore
\begin{eqnarray*}
F''&=&(4\pi)^{-\frac{n}{2}}\int_{\Sigma}[<\theta,\triangle_H\theta>+\frac{1}{2}\theta(x^\top)d^*\theta+\frac{1}{4}|\theta(x^\top)|^2-|\theta|^2]
e^{-\frac{|x|^2}{4}}d\mu
\\&&+(4\pi)^{-\frac{n}{2}}\int_{\Sigma}[-\theta(Jy^\perp)-h\theta(Jx^\perp)-\frac{1}{2}|y^\perp|^2-\frac{1}{4}h^2|x^\perp|^2]e^{-\frac{|x|^2}{4}}d\mu.
\end{eqnarray*}
By the decomposition (\ref{Hodgedecomposition}),
\begin{eqnarray*}
F''&=&(4\pi)^{-\frac{n}{2}}\int_{\Sigma}[<\theta,\triangle_H(du)>+\frac{1}{2}\theta(x^\top)d^*du+\frac{1}{4}|\theta(x^\top)|^2-|\theta|^2]
e^{-\frac{|x|^2}{4}}d\mu
\\&&+(4\pi)^{-\frac{n}{2}}\int_{\Sigma}[-\theta(Jy^\perp)-h\theta(Jx^\perp)-\frac{1}{2}|y^\perp|^2-\frac{1}{4}h^2|x^\perp|^2]e^{-\frac{|x|^2}{4}}d\mu.
\end{eqnarray*}
By integration by parts, we get
\begin{eqnarray*}
&&\int_{\Sigma}[<\theta,\triangle_H(du)>+\frac{1}{2}\theta(x^\top)d^*du+\frac{1}{4}|\theta(x^\top)|^2-|\theta|^2]
e^{-\frac{|x|^2}{4}}d\mu
\\&=&\int_{\Sigma}[|d^*du|^2+\theta(x^\top)d^*du+\frac{1}{4}|\theta(x^\top)|^2-|\theta|^2]e^{-\frac{|x|^2}{4}}d\mu
\\&=&\int_{\Sigma}[|d^*du+\frac{1}{2}\theta(x^\top)|^2-|\theta|^2]e^{-\frac{|x|^2}{4}}d\mu.
\end{eqnarray*}
Hence
\begin{eqnarray*}
F''&=&(4\pi)^{-\frac{n}{2}}\int_{\Sigma}[|d^*du+\frac{1}{2}du(x^\top)+\frac{1}{2}\theta_0(x^\top)|^2-|\theta|^2]e^{-\frac{|x|^2}{4}}d\mu
\\&&+(4\pi)^{-\frac{n}{2}}\int_{\Sigma}[-\theta(Jy^\perp)-h\theta(Jx^\perp)-\frac{1}{2}|y^\perp|^2-\frac{1}{4}h^2|x^\perp|^2]e^{-\frac{|x|^2}{4}}d\mu.
\end{eqnarray*}
\end{proof}

It is necessary to introduce the twisted Hodge Laplacian and the twisted Hodge decomposition theorem on compact Riemannian manifolds.
For a given smooth function $f$ on a compact manifold $(M,g)$, let $L_f^2$ be the space of those differential forms
which are square integrable with respect to the measure $e^{-f}d\mu_g$.
Let $d_f^*$ be the adjoint operator of $d$ in the Hilbert space $L_f^2$.
Then one has the twisted Hodge Laplacian
$$\triangle_{H,f}=d_f^*d+dd_f^*.$$
For the twisted Hodge Laplacian, Bueler \cite{Bueler} proved a Hodge decomposition theorem which states that
the space of $p$-forms has an orthogonal decomposition in $L^2_f$ by
$$\Omega^p=\mathcal{H}_f^p\oplus im d\oplus im d_f^*,$$
here $\mathcal{H}_f^p$ is the space of twisted harmonic $p$-forms, i.e. $p$-forms in the kernel of $\triangle_{H,f}$.
Hence $\mathcal{H}_f^p\cong H_{dR}(M)$ and for any closed $p$-form $\omega$ there exists a $(p-1)$-form $\alpha$ such that
$\omega+d\alpha$ is a twisted harmonic $p$-form. In particular for a $1$-form $\omega$, there exists a function $v$ such that
$d_f^*(\omega+dv)=0$.

We now come back to our situation.
Let $\Sigma$ be a closed Lagrangian self-shrinker that becomes extinct at $(0,1)$ and $f=\frac{|x|^2}{4}$.
Note that on one forms, $d_f^*=d^*+i_{\nabla f}$.
Hence for the $\theta=\theta_0+du$ in (\ref{Hodgedecomposition}),
$$d_f^*\theta=(d^*+i_{\frac{1}{2}x^\top})(\theta_0+du)=d^*du+\frac{1}{2}du(x^\top)+\frac{1}{2}\theta_0(x^\top).$$
Therefore we can rewrite the second variation formula ({\ref{secondvariationformulaLag}}) as follows.

\begin{corollary}\label{secondvariationLag2}
Let $\Sigma$ be a closed Lagrangian self-shrinker that becomes extinct at $(0,1)$ and $X$ a Lagrangian variation.
Let $\theta=-i_X\overline{\omega}$.
Then the second variation is given by
\begin{equation}\label{secondvariationformulaLag2}
F''=(4\pi)^{-\frac{n}{2}}\int_{\Sigma}[|d_f^*\theta|^2-|\theta|^2-\theta(Jy^\perp)-h\theta(Jx^\perp)-\frac{1}{2}|y^\perp|^2-\frac{1}{4}h^2|x^\perp|^2]
e^{-\frac{|x|^2}{4}}d\mu.
\end{equation}
\end{corollary}

\begin{lemma}\label{uzero}
For any harmonic $1$-form $\theta_0$, there exists a function $u_0$ such that
\begin{equation}\label{uzerofunction}
d^*du_0+\frac{1}{2}du_0(x^\top)+\frac{1}{2}\theta_0(x^\top)=0.
\end{equation}
\end{lemma}

\begin{proof}
Applying Bueler's twisted Hodge decomposition theorem to the harmonic $1$-form $\theta_0$, we see that there exists a function $u_0$ such that
$$0=d_f^*(\theta_0+du_0)=(d^*+i_{\frac{1}{2}x^\top})(\theta_0+du_0)=d^*du_0+\frac{1}{2}du_0(x^\top)+\frac{1}{2}\theta_0(x^\top).$$
\end{proof}

\begin{proposition}
Let $\Sigma$ be a closed Lagrangian self-shrinker that becomes extinct at $(0,1)$.
Let $\theta_0$ be any harmonic $1$-form and $u_0$ a solution of (\ref{uzerofunction}).
Then for the twisted harmonic $1$-form $\theta=\theta_0+du_0$, we have
\begin{equation}\label{secondvariationoftwistedharmonic}
F''=(4\pi)^{-\frac{n}{2}}\int_{\Sigma}[-|\theta+hH|^2-\frac{1}{2}|y^\perp|^2]e^{-\frac{|x|^2}{4}}d\mu,
\end{equation}
here $H$ is the mean curvature form $-i_H\overline{\omega}=-\frac{1}{2}<x,\nu_k>e^k$.
\end{proposition}

\begin{proof}
By (\ref{uzerofunction}), for $\theta=\theta_0+du_0$ we have
\begin{eqnarray*}
F''&=&(4\pi)^{-\frac{n}{2}}\int_{\Sigma}
[-|\theta|^2-\theta(Jy^\perp)-h\theta(Jx^\perp)-\frac{1}{2}|y^\perp|^2-\frac{1}{4}h^2|x^\perp|^2]e^{-\frac{|x|^2}{4}}d\mu.
\end{eqnarray*}
We now prove that the term $\int_{\Sigma}\theta(Jy^\perp)e^{-\frac{|x|^2}{4}}d\mu$ vanishes.
In fact
\begin{eqnarray*}
\int_{\Sigma}\theta(Jy^\perp)e^{-\frac{|x|^2}{4}}d\mu&=&\int_{\Sigma}\theta(e_k)<Jy,e_k>e^{-\frac{|x|^2}{4}}d\mu
\\&=&\int_{\Sigma}\theta(e_k)e_k<Jy,x>e^{-\frac{|x|^2}{4}}d\mu
\\&=&\int_{\Sigma}(d_f^*\theta)<Jy,x>e^{-\frac{|x|^2}{4}}d\mu
\\&=&0.
\end{eqnarray*}
Then for the twisted harmonic $1$-form $\theta=\theta_0+du_0$, we have
\begin{eqnarray*}F''(y,h,\theta)
&=&(4\pi)^{-\frac{n}{2}}\int_{\Sigma}[-|\theta|^2-h\theta(Jx^\perp)-\frac{1}{2}|y^\perp|^2-\frac{1}{4}h^2|x^\perp|^2]e^{-\frac{|x|^2}{4}}d\mu
\\&=&(4\pi)^{-\frac{n}{2}}\int_{\Sigma}[-\sum_k(\theta_k-\frac{1}{2}h<x,\nu_k>)^2-\frac{1}{2}|y^\perp|^2]
\\&=&(4\pi)^{-\frac{n}{2}}\int_{\Sigma}[-|\theta+hH|^2-\frac{1}{2}|y^\perp|^2].
\end{eqnarray*}
\end{proof}

\begin{remark}\label{Remark}
The mean curvature form $H=H_ke^k$ is twisted harmonic. In fact
$$d_f^*H=(d^*+i_{\frac{1}{2}x^\top})(-\frac{1}{2}<x,\nu_k>e^k)=0.$$
By the fact that a closed self-shrinker has non-vanishing mean curvature and $H$ is twisted harmonic,
one sees that the Maslov class $[H]\neq 0$.
Hence any closed Lagrangian self-shrinker must have $b_1\geq 1$.
This was first observed by Smoczyk, see \cite{CastroLerma2}.
\end{remark}

\begin{theorem}\label{Main2}
Any closed Lagrangian self-shrinker with $b_1\geq 2$ is Lagrangian F-unstable.
\end{theorem}

\begin{proof}
If $b_1\geq 2$, we can take a nontrivial class $[\theta_0]$ which is different from the class $[H]$.
Let $\theta=\theta_0+du_0$ be the twisted harmonic $1$-form in $[\theta_0]$.
Then it follows from (\ref{secondvariationoftwistedharmonic}) that $F''(y,h,\theta)<0$ for all $y, h$.
\end{proof}

By a theorem of Whitney any closed Lagrangian embedding $\Sigma$ in $\mathbb{C}^n$
must have vanishing Euler characteristic $\chi(\Sigma)$ (cf. \cite{Smoczyk4}).
If $n=2$, the only closed Lagrangian embedding in $\mathbb{C}^2$ is torus.

\begin{corollary}
In $\mathbb{C}^2$, there are no embedded closed Lagrangian F-stable Lagrangian self-shrinkers.
\end{corollary}

\section{Hamiltonian F-stability of closed Lagrangian self-shrinkers}

In this section we first give the second variation formula of the F-functional under Hamiltonian deformations.
We then show that for a closed Lagrangian self-shrinker $\Sigma$ with $b_1=1$,
the Lagrangian F-stability of $\Sigma$ is equivalent to the Hamiltonian F-stability of $\Sigma$.
Finally, we characterize the Hamiltonian F-stability of a closed Lagrangian self-shrinker
by its spectral property of the twisted Laplacian $\triangle_f$.

\begin{proposition}\label{secondvariationofHamiltonian}
Let $\Sigma$ be a closed Lagrangian self-shrinker that becomes extinct at $(0,1)$
and $X$ a Hamiltonian variation with $-i_X\overline{\omega}=du$.
Then the second variation is given by
\begin{eqnarray}\label{secondvariationformulaHam}
F''&=&(4\pi)^{-\frac{n}{2}}\int_{\Sigma}[|d_f^*du|^2-|du|^2-du(Jy^\perp)-\frac{1}{2}|y^\perp|^2-\frac{1}{4}h^2|x^\perp|^2]e^{-\frac{|x|^2}{4}}d\mu.
\end{eqnarray}
In particular $\Sigma$ is Hamiltonian F-stable if and only if for any Hamiltonian variation $J\nabla u$, there exists a vector $y$ such that
$$\int_{\Sigma}[|d_f^*du|^2-|du|^2-du(Jy^\perp)-\frac{1}{2}|y^\perp|^2]e^{-\frac{|x|^2}{4}}d\mu\geq 0.$$
\end{proposition}

\begin{proof}
It follows from (\ref{secondvariationformulaLag2}) and
\begin{eqnarray*}
\int_{\Sigma}du(Jx^\perp)e^{-\frac{|x|^2}{4}}d\mu
&=&\int_{\Sigma}u_k<Jx,e_k>e^{-\frac{|x|^2}{4}}d\mu
\\&=&\int_{\Sigma}-u_k<x,\nu_k>e^{-\frac{|x|^2}{4}}d\mu
\\&=&\int_{\Sigma}u[<x,-H_pe_p>-\frac{1}{2}<x,\nu_k><x,e_k>]e^{-\frac{|x|^2}{4}}d\mu
\\&=&0.
\end{eqnarray*}
\end{proof}

We now consider the case that $b_1(\Sigma)=1$.
Note that the mean curvature form $H$ represents the nontrivial Maslov class.
Hence any closed $1$-form $\theta$ can be written as
$$\theta=-2aH+du.$$

\begin{theorem}
For a closed Lagrangian self-shrinker $\Sigma$ with $b_1=1$, the Lagrangian F-stability of $\Sigma$ is equivalent to
the Hamiltonian F-stability of $\Sigma$.
\end{theorem}

\begin{proof}
Let $X$ be any Lagrangian variation and
$$\theta=-i_X\overline{\omega}=-2aH+du=(a<x,\nu_k>+u_k)e^k.$$
By the formula (\ref{secondvariationformulaLag2}) and the fact that the mean curvature form $H$ is twisted harmonic,
\begin{eqnarray*}
F''&=&(4\pi)^{-\frac{n}{2}}\int_{\Sigma}[|d^*_fdu|^2-|\theta|^2-\theta(Jy^\perp)-h\theta(Jx^\perp)-\frac{1}{2}|y^\perp|^2-\frac{1}{4}h^2|x^\perp|^2]
e^{-\frac{|x|^2}{4}}d\mu
\\&=&(4\pi)^{-\frac{n}{2}}\int_{\Sigma}[|d^*_fdu|^2-(|du|^2+a^2|x^\perp|^2-2adu(Jx^\perp))]e^{-\frac{|x|^2}{4}}d\mu
\\&&+(4\pi)^{-\frac{n}{2}}\int_{\Sigma}-[a<x,\nu_k><Jy^\perp,e_k>+du(Jy^\perp)]e^{-\frac{|x|^2}{4}}d\mu
\\&&+(4\pi)^{-\frac{n}{2}}\int_{\Sigma}-h[a<x,\nu_k><Jx^\perp,e_k>+du(Jx^\perp)]e^{-\frac{|x|^2}{4}}d\mu
\\&&+(4\pi)^{-\frac{n}{2}}\int_{\Sigma}[-\frac{1}{2}|y^\perp|^2-\frac{1}{4}h^2|x^\perp|^2]e^{-\frac{|x|^2}{4}}d\mu.
\end{eqnarray*}
In the proof of Proposition \ref{secondvariationofHamiltonian}, we have seen that
$$\int_{\Sigma}du(Jx^\perp)e^{-\frac{|x|^2}{4}}d\mu=0.$$
Then
\begin{eqnarray*}
F''&=&(4\pi)^{-\frac{n}{2}}\int_{\Sigma}[|d^*_fdu|^2-(|du|^2+a^2|x^\perp|^2)]e^{-\frac{|x|^2}{4}}d\mu
\\&&+(4\pi)^{-\frac{n}{2}}\int_{\Sigma}-[a<x,\nu_k><Jy^\perp,e_k>+du(Jy^\perp)]e^{-\frac{|x|^2}{4}}d\mu
\\&&+(4\pi)^{-\frac{n}{2}}\int_{\Sigma}ah|x^\perp|^2e^{-\frac{|x|^2}{4}}d\mu
\\&&+(4\pi)^{-\frac{n}{2}}\int_{\Sigma}[-\frac{1}{2}|y^\perp|^2-\frac{1}{4}h^2|x^\perp|^2]e^{-\frac{|x|^2}{4}}d\mu.
\end{eqnarray*}
By Lemma \ref{Lemma}, we have
\begin{eqnarray*}
\int_{\Sigma}<x,\nu_k><Jy^\perp,e_k>e^{-\frac{|x|^2}{4}}d\mu
&=&\int_{\Sigma}-<Jx^\perp,Jy^\perp>e^{-\frac{|x|^2}{4}}d\mu
\\&=&\int_{\Sigma}-<x^\perp,y>e^{-\frac{|x|^2}{4}}d\mu
\\&=&0.
\end{eqnarray*}
Hence
\begin{eqnarray*}
F''&=&(4\pi)^{-\frac{n}{2}}\int_{\Sigma}[|d^*_fdu|^2-|du|^2-du(Jy^\perp)-\frac{1}{2}|y^\perp|^2]e^{-\frac{|x|^2}{4}}d\mu
\\&&+(4\pi)^{-\frac{n}{2}}\int_{\Sigma}-(a-\frac{1}{2}h)^2|x^\perp|^2e^{-\frac{|x|^2}{4}}d\mu.
\end{eqnarray*}
Comparing this second variation formula with (\ref{secondvariationformulaHam}), we get the equivalence.
\end{proof}

Let
$$\lambda_1=\inf\frac{\int_\Sigma|d\varphi|^2e^{-\frac{|x|^2}{4}}d\mu}{\int_\Sigma|\varphi|^2e^{-\frac{|x|^2}{4}}d\mu},$$
where the infimum is taken over all non-zero $\varphi$ with $\int_\Sigma \varphi e^{-\frac{|x|^2}{4}}d\mu=0$.
Then the first eigenfunction $\varphi_1$ satisfies
$$d^*d\varphi_1+\frac{1}{2}d\varphi_1(x^\top):=-\triangle_f\varphi_1=\lambda_1\varphi_1.$$
Let $\lambda_1<\lambda_2<\lambda_3<\cdots$ be the nonzero eigenvalues of $\triangle_f$.

\begin{theorem}\label{HamiltonianFStable}
A closed Lagrangian self-shrinker $\Sigma$ that becomes extinct at $(0,1)$ is Hamiltonian F-stable if and only if the twisted Laplacian $\triangle_f$ has
$$\lambda_1=\frac{1}{2}, \quad \Lambda_{\frac{1}{2}}=\{<x,w>, w\in \mathbb{R}^{2n}\}; \quad \lambda_2\geq 1.$$
\end{theorem}

\begin{proof}
Note that for any vector $w\in \mathbb{R}^{2n}$,
$$\triangle_f <x,w>=-\frac{1}{2}<x,w>,$$
hence $$0<\lambda_1\leq \frac{1}{2}.$$
The formula (\ref{secondvariationformulaHam}) can be rewritten as
$$F''=(4\pi)^{-\frac{n}{2}}\int_{\Sigma}[|d_f^*du|^2-\frac{1}{2}|du|^2-\frac{1}{2}|\nabla u+Jy^\perp|^2-\frac{1}{4}h^2|x^\perp|^2]e^{-\frac{|x|^2}{4}}d\mu.$$
If $\lambda_1<\frac{1}{2}$, by taking $u=\varphi_1$ we get
\begin{eqnarray*}
F''&=&(4\pi)^{-\frac{n}{2}}\int_{\Sigma}\lambda_1(\lambda_1-\frac{1}{2})\varphi_1^2e^{-\frac{|x|^2}{4}}d\mu
\\&&+(4\pi)^{-\frac{n}{2}}\int_{\Sigma}[-\frac{1}{2}|\nabla \varphi_1+Jy^\perp|^2-\frac{1}{4}h^2|x^\perp|^2]e^{-\frac{|x|^2}{4}}d\mu
\\&<&0.
\end{eqnarray*}
Hence $\lambda_1=\frac{1}{2}$ is necessary for the Hamiltonian F-stability.

For any first eigenfunction $\varphi_1$ with $-\triangle_f\varphi_1=\frac{1}{2}\varphi_1$, we get
\begin{eqnarray*}
F''&=&(4\pi)^{-\frac{n}{2}}\int_{\Sigma}[-\frac{1}{2}|\nabla \varphi_1+Jy^\perp|^2-\frac{1}{4}h^2|x^\perp|^2]e^{-\frac{|x|^2}{4}}d\mu.
\end{eqnarray*}
Note that
$$Jy^\perp=(Jy)^\top=\nabla<x,Jy>,$$
hence the following is necessary for the Hamiltonian F-stability
$$\varphi_1+<x,Jy>=0.$$
Therefore Hamiltonian F-stability implies that the first eigenfunction space is
$$\Lambda_{\frac{1}{2}}=\{<x,w>, w\in \mathbb{R}^{2n}\}.$$
Here $<x,w>=0$ may happen for some $w\neq 0$.
Note also that the corresponding Hamiltonian variations generated by the first eigenfunction space are the translations $w^\perp$.

Assuming that $\lambda_1=\frac{1}{2}$ and the first eigenfunction space is $\{<x,w>, w\in \mathbb{R}^{2n}\}$, we now show that
the Hamiltonian F-stability is equivalent to $\lambda_2\geq 1$.
Given a function $u$, it admits a decomposition
$$u=a+<x,w>+u_2,$$
such that
$$\int_\Sigma u_2e^{-\frac{|x|^2}{4}}d\mu=\int_\Sigma u_2<x,z>e^{-\frac{|x|^2}{4}}d\mu=0,\quad \forall z\in \mathbb{R}^{2n}.$$
Then by (\ref{secondvariationformulaHam}) and the above orthogonal condition,
\begin{eqnarray*}
F''&=&(4\pi)^{-\frac{n}{2}}\int_{\Sigma}[\frac{1}{4}<x,w>^2+|d_f^*du_2|^2-\frac{1}{2}<x,w>^2-|du_2|^2]e^{-\frac{|x|^2}{4}}d\mu
\\&&+(4\pi)^{-\frac{n}{2}}\int_{\Sigma}[-<Jy^\perp,w>-du_2(Jy^\perp)-\frac{1}{2}|y^\perp|^2-\frac{1}{4}h^2|x^\perp|^2]e^{-\frac{|x|^2}{4}}d\mu.
\end{eqnarray*}
By Lemma \ref{Lemma}
$$\int_{\Sigma}<x,w>^2e^{-\frac{|x|^2}{4}}d\mu=\int_{\Sigma}2|w^\top|^2e^{-\frac{|x|^2}{4}}d\mu,$$
and by the above mentioned orthogonal condition,
\begin{eqnarray*}
\int_{\Sigma}du_2(Jy^\perp)e^{-\frac{|x|^2}{4}}d\mu
&=&\int_{\Sigma}<du_2,d<Jy,x>>e^{-\frac{|x|^2}{4}}d\mu
\\&=&\int_{\Sigma}u_2d_f^*d<Jy,x>e^{-\frac{|x|^2}{4}}d\mu
\\&=&\int_{\Sigma}u_2\frac{1}{2}<Jy,x>e^{-\frac{|x|^2}{4}}d\mu
\\&=&0.
\end{eqnarray*}
Hence
\begin{eqnarray*}
F''&=&(4\pi)^{-\frac{n}{2}}\int_{\Sigma}[|d_f^*du_2|^2-|du_2|^2-\frac{1}{2}|w^\top+Jy^\perp|^2-\frac{1}{4}h^2|x^\perp|^2]e^{-\frac{|x|^2}{4}}d\mu.
\end{eqnarray*}
Taking $y=Jw$ and $h=0$, we get
\begin{eqnarray*}
F''&=&(4\pi)^{-\frac{n}{2}}\int_{\Sigma}[|d_f^*du_2|^2-|du_2|^2]e^{-\frac{|x|^2}{4}}d\mu,
\end{eqnarray*}
which is nonnegative for all such $u_2$ if and only if $\lambda_2\geq 1$.
\end{proof}

From the proof of Theorem \ref{HamiltonianFStable}, we see the following

\begin{theorem}
A closed Lagrangian self-shrinker $\Sigma$ that becomes extinct at $(0,1)$ is Hamiltonian F-stable if and only if
\begin{equation*}
\int_\Sigma|du|^2e^{-\frac{|x|^2}{4}}d\mu\geq \int_\Sigma u^2e^{-\frac{|x|^2}{4}}d\mu, \quad \text{for all u s.t. }\quad
\int_\Sigma ue^{-\frac{|x|^2}{4}}d\mu=\int_\Sigma uxe^{-\frac{|x|^2}{4}}d\mu=0.
\end{equation*}
\end{theorem}

We have seen that the Lagrangian F-stability of a closed self-shrinker with $b_1=1$ is equivalent to the Hamiltonian F-stability.
However for $b_1\geq 2$, there is a difference between two stabilities.
A simple example is the (Lagrangian) Clifford torus $$T^n=\{(z^1,\cdots,z^n): \quad |z^1|^2=\cdots=|z^n|^2=2\},$$
whose mean curvature is $H=-\frac{1}{2}x=-\frac{1}{2}x^\perp$.
Hence the Clifford torus is a Lagrangian self-shrinker that becomes extinct at $(0,1)$.
By Theorem \ref{Main2}, the Clifford torus is Lagrangian F-unstable.
On the other hand it is well-known that the first non-zero eigenvalue of $T^n$ is $\lambda_1(\triangle)=\frac{1}{2}$,
and the first eigenfunction space is spanned by $\{x^k, y^k\}_{k=1}^n$; the second eigenvalue is $\lambda_2(\triangle)=1$.
Hence the Clifford torus is Hamiltonian F-stable.

\end{document}